\documentclass[11pt]{article}
\usepackage{latexsym}              

\usepackage{amssymb,amsmath, amsrefs,amsthm}           
\usepackage[english]{babel}
\usepackage{graphicx}
\usepackage{setspace}
\usepackage{epsfig}
\usepackage{graphics}
\usepackage{cite}
\usepackage[justification=centering]{caption}
\usepackage{multicol}
\usepackage{mathrsfs}
\usepackage{pgf,tikz}
\usetikzlibrary{arrows}
\usepackage{float}
\restylefloat*{figure}

\usepackage{blkarray}
\usepackage[utf8]{inputenc}

\newtheorem{theorem}{Theorem}[section]
\newtheorem{lemma}[theorem]{Lemma}
\newtheorem{proposition}[theorem]{Proposition}
\newtheorem{corollary}{Corollary}
\newtheorem{definition}{Definition}
\newtheorem{example}{Example}
\newtheorem{question}{Question}
\newtheorem{conjecture}{Conjecture}

\newcommand{\LGR}{{{}^\mathcal{L}G^\mathcal{R}}}
\newcommand{\LG}{{{}^\mathcal{L}G}}
\newcommand{\GR}{{G^\mathcal{R}}}
\newcommand{\LHR}{{{}^\mathcal{L}H^\mathcal{R}}}
\newcommand{\LH}{{{}^\mathcal{L}H}}
\newcommand{\HR}{{H^\mathcal{R}}}

\newcommand{\SQ}{{SQ}}
\usepackage{color}
\usepackage{tikz}
\usetikzlibrary{matrix,decorations.pathreplacing,calc}

\pgfkeys{tikz/mymatrixenv/.style={decoration=brace,every left delimiter/.style={xshift=3pt},every right delimiter/.style={xshift=-3pt}}}
\pgfkeys{tikz/mymatrix/.style={matrix of math nodes,left delimiter=[,right delimiter={]},inner sep=2pt,column sep=1em,row sep=0.5em,nodes={inner sep=0pt}}}
\pgfkeys{tikz/mymatrixbrace/.style={decorate,thick}}

\begin{document}

\title{Simultaneous Combinatorial Game Theory\thanks{The first author is supported by the Natural Sciences and Engineering Research Council of Canada and the Killam Trust. The second author is supported by the Natural Sciences and Engineering Research Council of Canada.}}

\author{Melissa Huggan and
        Richard Nowakowski, Dalhousie University\\
               Paul Ottaway,
                  Capilano University
}


\maketitle

\begin{abstract}

Combinatorial game theory (CGT), as introduced by Berlekamp, Conway \& Guy, involves two players who move alternately in a perfect
  information, zero-sum 
 game, and there are no chance devices. 
 Also the games have the finite descent property (every game terminates in a finite number of moves). 
The two players are usually called \textit{Left} and \textit{Right}.

The games often break up into components and the players must choose
one of the components in which to play. 
One main aim of CGT is to analyze the components individually (rather than 
analyzing the sum as a whole) then use this information to analyze the sum. 

In this paper, the players move simultaneously in a combinatorial game. Three sums are considered which are defined by the termination rules: (i) one component does not have a simultaneous move; (ii) no component has a simultaneous move; 
(iii) one player has no move in any component.
These are combined with a winning convention which is either: (i) based on which player has moves remaining; or (ii) the greatest score.
 In each combination, we show that equality of games induces an equivalence relation and the equivalence classes are partially ordered.
 Also, where possible, given games $A$ and $B$, we give checks to determine if Left prefers to replace $A$ by $B$ in a sum.

Keywords: Combinatorial Game Theory, Economic Game Theory, Simultaneous Combinatorial Game Theory, Disjunctive Sum, Conjunctive Sum, Continued Conjunctive Sum, Extended Normal Play and Scoring Play.
\end{abstract}

\section{Introduction}\label{intro}

Combinatorial game theory (CGT), as introduced by Berlekamp, Conway \& Guy \cite{Conway1976,BerleCG}
 (see also \cite{Albert2007,Siegel2013}), involves two players who move alternately in a perfect
  information, zero-sum
 game, with no chance devices, that has the finite descent property (every game terminates in a finite number of moves). The two players are usually 
 called \textit{Left} and \textit{Right}, where Left is female and Right is male.
 The \textit{normal} play winning convention has the last player to move as the winner. Recently, the theory has been extended to include games where 
 the winner is determined by a \textit{score} rather than who moves last \cite{Larsson2015}.

  In this paper, we examine  combinatorial games when players move simultaneously instead of alternately. In CGT, given a position $G$, the \textit{options} are those positions that can be reached in one move.
Left and Right options are denoted by the sets $G^\mathcal{L}$ and $G^\mathcal{R}$ respectively.
 The position $G$ can be written as $\{G^\mathcal{L}\mid G^\mathcal{R}\}$.
 In simultaneous play, a position, therefore, already has a set of Left options and Right options
 and we consider these as the basic 
(individual) moves. The rules and winning conventions must be extended to cover 
the simultaneous moves. \\

 In CGT, since these are games of pure strategy, the outcome of each game is determined,
  denoted by $o(G)$.
 There are four outcomes possible for a game:
 \[o(G)=\begin{cases}
 \mathcal{L}, \text{Left can force a win regardless of who starts;}\\
 \mathcal{R}, \text{Right can force a win regardless of who starts;}\\
  \mathcal{N}, \text{the Next player can force a win;}\\
   \mathcal{P}, \text{the Previous player can force a win.}\\
 \end{cases}
 \]
Since these games are zero-sum, two player games, the usual convention is followed, 
 that is, Left wins are positive and
  Right negative.
 The outcomes are partially ordered: from Left's point of view, she prefers $\mathcal{L}$
 over both $\mathcal{P}$ and $\mathcal{N}$, which are incomparable, and she prefers these over $\mathcal{R}$.
 
CGT is sometimes referred to as \emph{additive} game theory because such games often decompose
into components and a player must choose one and play in it. 
(For Maker-Maker games, for example, connection games like \textsc{hex}, see Beck~\cite{Beck}.)
This is formalized in the notion of disjunctive sum.

 \begin{definition}
The disjunctive sum of games $G$ and $H$ is 
\begin{center}
$G + H = \left\{G^{\mathcal{L}}+H, G+H^{\mathcal{L}} \mid G^{\mathcal{R}}+H, G+H^{\mathcal{R}}\right\}$.
\end{center}
\end{definition}

In normal play, the disjunctive sum $G+H$ can be analyzed by considering 
$G$ and $H$ separately and then combining the results rather than having to 
consider the total game \cite{BerleCG,Conway1976}. 
Now, play must alternate in $G+H$ but this 
alternation need not happen in $G$ and $H$ separately. It is possible that Left is happy to 
play three consecutive moves in $G$ whilst, each time,
 Right responds in $H$. No other sum creates as useful a
 structure as that of normal play and disjunctive sum. (Also see the extension to scoring games \cite{Larsson2015,Larsson2017}.)

The disjunctive sum also leads to the concept of equality (indifference) and produces a partial order on games. 
 \begin{def}\label{def: equality} Given positions $G$ and $H$,

\noindent \textit{Equality:} $G = H$ if $(\forall X)$  $o(G+ X)=o(H+X)$.

\noindent
\textit{Greater than:} $G \geq H$ if $(\forall X)$ $o(G+X)\geq o(H+X)$. 
\end{def}

The defined `$=$' relation is an equivalence relation on CGT games and the quotient is a partially 
ordered abelian group. \\

We are interested in simultaneous games that have components, how these components are played,
and extending the concepts of equality and inequality. Conway \cite{Conway1976} Chapter 14, defines a dozen different ways of combining positions. (See also \cite{BerleCG} Chapter 9, and \cite{FraenT,Guignard2009,Kane2010,Smith1966}).
We extend just three.
Consider the positions $G$ and $H$. In the \textit{disjunctive} sum, players must play in either $G$ or $H$. In the \textit{conjunctive} sum, the players must play in all components. Hence, in the conjunctive sum of $G$ and $H$, if a player has a move in $G$ but not in $H$ (or vice versa) the game is over. Lastly, in the \textit{continued conjunctive} sum, the players must play in all components where both players still have moves remaining.  

In simultaneous games, since the players move at the same time, the winning conventions cannot be based purely on who moves last, as in CGT. We define two conventions. The first depends on a \textit{non-losing condition}: if a player has a move in a component $C$ then that player cannot lose the whole game because of $C$. If, in a sum, a player has a move in 
each component then that player cannot lose on that turn. 
The second winning convention is determined by a score which is assigned at the end of the game, i.e., at a \textit{terminal} position. 
Both conventions allow a \textit{Draw} as an outcome.

There are three possible outcomes of a terminal position of a simultaneous combinatorial game: 
  \[o_{S}(G)=\begin{cases}
 \mathcal{L}, \text{Left wins;}\\
 \mathcal{R}, \text{Right wins;}\\
  \mathcal{D}, \text{Draw otherwise.}\\
 \end{cases}
 \]
In keeping with the two-player, zero-sum conventions, we will order the outcomes $\mathcal{L}>\mathcal{D}>\mathcal{R}$.
\\

 Given a game $G$, in CGT, there are operations that, when applied repeatedly, result in a game $H$ where (i) $G=H$; (ii) $H$ has the smallest game tree of all games equal to $G$; and (iii) $o(G+X) = o(H+X)$. One of these operations is to eliminate dominated, or one of two equal, strategies\footnote{The other, reversing reversible options, is particular to alternating play, and has no analogue in simultaneous play.}. However, eliminating one of two equal strategies can cause problems in sums of simultaneous games.

In Section~\ref{basic concepts}, we introduce the concepts required to analyze the games, including evaluations of \emph{expected value}. 
When analyzing a two-player, zero-sum game, dominated strategies can be eliminated without changing the expected value.
In Section~\ref{sums of simultaneous games: normal}, we consider the different sums under the non-losing condition. In each case, we show that the expected value of
 the sum of $G$ and $H$ is not the same as eliminating  dominated strategies in $G$ and in $H$ then taking the sum.
 It is an open question as to what reductions may be applied to $G$ and $H$ so that the calculations can be simplified.
 In Section~\ref{sums of simultaneous games: scoring} we consider all three sums but now under the scoring convention. 
 The continued conjunctive sum has an easy test for equality and inequality based on the expected value, Theorem~\ref{thm: expected value scoring}. Lastly, in Section~\ref{sec: case studies}, we investigate three case studies \textsc{simultaneous clobber}, \textsc{simultaneous hackenbush}, and \textsc{subtraction squares} to demonstrate the previously defined concepts. \\

Work in this area began in 2007 under the description of \emph{synchronized games}\footnote{Cincotti et al. used synchronized to describe games where the moves
were of a particular type. Our scope is more general so we use the term \textit{simultaneous}.}. Cincotti and Iida~\cite{Cincotti2007} studied \textsc{cutcake} under simultaneous moves using disjunctive sum but with different outcome classes than we consider. In 2008, Cincotti and Iida~\cite{Cincotti2008a} studied \textsc{synchronized domineering} and solved the outcome classes for 
several board sizes. These proof ideas have been extended for work on \textsc{synchronized triomineering}~\cite{Cincotti2008b,Cao2012}, 
\textsc{synchronized tridomineering}~\cites{Cincotti2008b}, \textsc{synchronized quadromineering}~\cites{Cincotti2010, Cincotti2012 }.
 Bahri and Kruskal~\cite{Bahri2010} presented a new method for considering \textsc{synchronized domineering} which bounds 
 the outcomes using combinatorial game theory techniques. However, no framework for general rulesets or game values is developed. 

 \section{Basic Concepts}\label{basic concepts}
Before looking at the individual sums, we introduce common concepts and a simple game useful for examples.

\begin{definition} (Ruleset) Given a set of game positions $\Omega$, a \textit{ruleset} over $\Omega$ consists of
three functions $L,R,S:\Omega\rightarrow 2^{\Omega}$. For $G\in \Omega$, $L(G)$ is the set of \textit{Left options},
which we will denote as $\LG$, 
 $R(G)$ is 
the set of \textit{Right options}, denoted $\GR$,  and $S(G)$ is the set of \textit{simultaneous options}, denoted $\LGR$.
Moreover, for each $H\in S(G)$ there exists $i,j$ such that $H$ is associated with $^{L_i}G$ and $G^{R_j}$. A game $G$ is called a \textit{terminal position} if $^{\mathcal{L}}G^{\mathcal{R}} = \emptyset$. 
\end{definition}

\begin{definition}
The game $G$ can be represented as a \emph{matrix}, $M(G)$. The elements of $\LG$ label the rows
 (pure strategies for Left), and elements of $\GR$ label the columns (pure strategies for Right). 
 An entry $^{L_{i}}G^{R_{j}}$ will be the result of Left playing pure strategy $i$ and Right playing pure strategy $j$ in
  $G$.\\
  
   Hence if Left has $m$ pure strategies and Right has $n$ pure strategies we have,
   
\begin{equation*}
M(G) =
\begin{blockarray}{ccccc}
&$$G^{R_{1}}$$&$$G^{R_{2}}$$&$$\,\,\ \ldots$ $&$$G^{R_{n}} $$ \\
\begin{block}{c[cccc]}
$$ ^{L_{1}}G$$&$ $  ^{L_{1}}G^{R_{1}}$$ &$$ ^{L_{1}}G^{R_{2}}$$&$ $\ldots$$ &$$^{L_{1}}G^{R_{n}}$$\\
$$ ^{L_{2}}G$$ &$$  ^{L_{2}}G^{R_{1}}$$&$$ ^{L_{2}}G^{R_{2}}$$&$ $\ldots$$&$$^{L_{2}}G^{R_{n}}$$\\
   $$\vdots$$&$$\vdots$$&$$\vdots$$&&$$\vdots$$&\\
$$^{L_{m}}G$$&$ $ ^{L_{m}}G^{R_{1}}$$&$$ ^{L_{m}}G^{R_{2}}$$&$$\ldots$$&$$^{L_{m}}G^{R_{n}}$$\\
\end{block}
\end{blockarray}.
 \end{equation*}
 \end{definition}
In many games, a position obtained by a Left move followed by a Right move can be 
reached by interchanging the moves.
In these cases the moves can be played simultaneously without further clarification. 

The ruleset must include a mechanism for determining the effect of the simultaneous move, since there is no general procedure for determining $\,^{L_{i}}G^{R_{j}}$ from $^{L_{i}}G$ and $G^{R_{j}}$. We require that the simultaneous moves retain \textit{perfect information} and \textit{finite descent}, or at least the expected number of moves is finite.

\subsection{Evaluations and Outcomes}\label{ss:eando}
\begin{definition}
 For the \textit{extended normal play winning convention}, the outcomes of a terminal position are defined as follows: 
 \[o_{S}(G)=\begin{cases}
 \mathcal{L}, \text{ if } \LG\ne\emptyset,\, \LGR=\emptyset= \GR, \\
 \mathcal{R}, \text{ if } \LG=\emptyset=\LGR, \, \GR\ne\emptyset,\\
  \mathcal{D}, \text{Draw otherwise.}\\
 \end{cases}
 \]
 \end{definition}
In other words, Left wins if she has moves remaining in $G$ and Right does not; Right wins if he has moves and Left does
not; and it is a Draw if neither player has moves. 
 
\bigskip
\noindent
Ruleset for \textsc{subtraction squares}, \SQ($S_{L}$, $S_{R}$) on a strip of squares of length $n$,
denoted \SQ$(S_{L}$, $S_{R})(\underline{n})$ . 
\begin{itemize}
\renewcommand{\labelitemi}{{\bf $\circ$}}
\item Board: : Let $S_L$ and $S_R$ be sets of positive integers. The board is a strip of $n$ squares
denoted $\underline{n}$.\\
\item Moves: For any $p\in S_L$, $p\leq n$, Left can remove $p$ squares from the left or right side of the strip. 
Similarly, if $q\in S_R$ then Right can remove $q$ squares from the left or right side.  \\
\item Simultaneous rule:  If they both take from the same side then $\max\{p,q\}$
squares are removed. If they take from opposite sides then the move is to $n-p-q$ 
except if $\max\{p,q\}\leq n\leq p+q$ then the move is to 0.\\
\end{itemize}

Note that if the simultaneous rule is always to remove $p+q$, without any reference to the side played, then in \SQ($\{1, 10\},\{2, 10\}(\underline{12})$ neither player knows if subtracting $10$ is legal. On the other hand, if the simultaneous rule is to remove $| p-q|$ then the same game could last forever. 

In \SQ($\{1\},\{2\}$), $\underline{0}$ is a Draw and $\underline{1}$ is a Left win, since Left has a move and Right doesn't. In $\underline{2}$, regardless of whether they play on the left (l) or right (r) the result is $\underline{0}$ which is a Draw. 

A useful game for further examples is \SQ$'(\{1\},\{2\})$, which is \SQ($\{1\},\{2\}$), except  Left is not allowed to move in $\underline{2}$. 
As in $SQ(\{1\},\{2\})$, $\underline{0}$ is a draw and $\underline{1}$ is a Left win but now $\underline{2}$ is a Right win. \\

Games which are won by scores are considered in Section \ref{sums of simultaneous games: scoring}.
Now, scores can be assigned to a terminal position in many possible ways. We follow a CGT approach.

Let $G$ be a CGT game where at least one player does not have a move.
Denote by $v_{A}(G)$ the CGT value of $G$.
 The value, $v_{A}(G)$, equals
the maximum number of moves one player can make before opening new moves for their opponent. Thus
$v_A(G)$ is an integer, non-negative if Right cannot move and non-positive if Left cannot move.
 For example, in Figure~\ref{example: disjunctive sum}, $H_{2}$ has a score of $2$. 
When $G$ is re-interpreted as a simultaneous game, under a particular sum, then $G$ is terminal
and each terminal component will have a CGT value associated with it. The 
score of $G$ will depend on the termination rules for that sum. 

For simultaneous play,  the usual interpretation would have that Left wins if $v_A(G)>0$, Right wins if $v_A(G)<0$
and a draw otherwise. \\

In Section~\ref{sums of simultaneous games: normal}, we show that the disjunctive sum requires the players to know 
the expected value of the game
under any winning condition we consider. Moreover, in Example \ref{ex:dsnp}, we show that the expected value
cannot form the basis of an evaluation function in general.

 In the sums that we consider, it may be possible for the two players to play in different components. Thus in a sum $G\odot H$, the definition of
$S(G\odot H)$ may require all of $\LG$, $\GR$, $\LGR$, $\LH$, $\HR$, and $\LHR$.

 We also consider both the conjunctive and continued conjunctive 
sums. In Section~\ref{sec: case studies}, we demonstrate, via case studies,
 how different sums and winning conditions affect the outcomes of game play.

A  natural tool to consider when presented with a matrix of games  is to assign values to terminal games and 
calculate the expected value of the game. Here, to represent the outcome
 in terms of expected values, we define two related measures. 
 \begin{definition} Let $G$ be a game. The \textit{expected value}, $Ex(G)$, is given recursively
 \begin{eqnarray*}
 Ex(G) &=& \begin{cases} 1, \text{ if $G$ is terminal and a Left win,}\\
                                      0, \text{ if $G$ is terminal and a Draw,}\\
                                      -1, \text{ if $G$ is terminal and a Right win,}\\
                                       Ex(M'(G)), \text{where $M'(G)$ is $M(G)$ in which each ${}^{L_i}G^{R_j}$ is}\\
                                       \qquad \qquad \qquad\text{replaced by $Ex({}^{L_i}G^{R_j})$}.
                    \end{cases}
 \end{eqnarray*}
 
 \end{definition}

We assume that a player wants to maximize their expectation of winning. Therefore, in $M'(G)$ we have the standard concepts of domination \cite{Barron2008} and thus dominated strategies can be eliminated thereby reducing the matrix. Eliminating those dominated options also translates back to $G$ where the corresponding options can also be eliminated. This is called the \textit{reduced} game, $Re(G)$.\\

In \SQ($\{1\},\{2\})(\underline{3})$, the options are given in Figure \ref{fig: 1/2 games}, outcomes in  Figure~\ref{fig: 1/2 outcomes}, and expected values of the options in Figure~\ref{fig:1/2}. From Figure~\ref{fig:1/2} we determine that the expected value of the overall game is $Ex(\underline{3})=1/2$.

 In $k$  plays of $\underline{3}$, since Right can never win, 
Left will expect
to win half of the games and the other half will be Draws.
  
  \begin{figure}[h]
  \begin{center}
 \begin{minipage}[b]{0.25\linewidth}
  \[
\begin{blockarray}{ccc}
&l & r  \\
\begin{block}{c[cc]}
 l & \underline{1} & \underline{0} \\
  r & \underline{0} & \underline{1} \\
\end{block}
\end{blockarray}
 \]
   \caption{Games.}\label{fig: 1/2 games}
 \end{minipage}
 \begin{minipage}[b]{0.25\linewidth}
  \[
\begin{blockarray}{ccc}
&l &r \\
\begin{block}{c[cc]}
 l & $L$ & $D$ \\
  r& $D$  & $L$ \\
\end{block}
\end{blockarray}
 \]
   \caption{Outcomes.}\label{fig: 1/2 outcomes}
 \end{minipage}
 \begin{minipage}[b]{0.3\linewidth}
  \[
\begin{blockarray}{ccc}
&l &r  \\
\begin{block}{c[cc]}
l& 1 & 0 \\
r& 0 & 1 \\
\end{block}
\end{blockarray}
 \]
  \caption{Expected Values.}\label{fig:1/2}
 \end{minipage}
\end{center}
\end{figure}

 The challenge with this
  approach is, if the game is a sum of two other games, say $G$ is the sum of $H$ and $K$, what reductions can be first applied to $H$ and $K$ individually before considering their sum. This is a similar problem previously considered in CGT. Even though games are played under alternating play, one cannot restrict the study of components 
 and insist that play alternates in each component.
 In the disjunctive sum of games $A$ and $B$, Left could prefer to play in $A$ and Right 
  could prefer to play in $B$. Hence, Left could have two or more consecutive moves in $A$ and Right two moves in $B$.
   
\section{Sums of Simultaneous Games: Extended Normal Play}\label{sums of simultaneous games: normal}

The notions of equality and greater than are defined similarly for the outcome classes of any generic sum, $\odot$.
\begin{definition}\label{defn:equality} Let $G$ and $H$ be games,
\begin{itemize}
\renewcommand{\labelitemi}{{\bf $\circ$}}
\item \textit{Equality of games:} $G = H$ if $(\forall X)$  $\alpha(G\odot X)=\alpha(H\odot X)$.
\item \textit{Greater than:} $G \geq H$ if $(\forall X)$,  $\alpha(G\odot X)\geq \alpha(H\odot X)$.
\end{itemize}
where $\alpha$ represents a generic measure of winning.
\end{definition}

Care must be taken when reducing weakly dominated strategies in $G$ and $H$, since $o_{S}(G\odot H)$ is not necessarily equal to $o_{S}(G'\odot H')$.\\

All the sums have two common properties.

\begin{theorem}
Simultaneous combinatorial games, under a sum, form an equivalence relation and the quotient is a
partial order.
\end{theorem}
\begin{proof}
From the definition of equality, it is clear that: (i) $G=G$ for all $G$; (ii) if $G=H$ then $H=G$; and (iii) if $G=H$ and $H=K$ then $G=K$.
Therefore equality is an equivalence relation.

Equal games are identified to obtain the quotient by `$=$', that is, the objects are now the equivalence classes.  The proof for a partial order is now similar
to that for equality.\qed
\end{proof}

An open question is what properties, if any, does the partial order have. In alternating play CGT, the order is a distributive lattice. Here we only know about 
the continued conjunctive sum with the scoring winning convention, see Corollary \ref{cor:ccspo}.
We consider a narrower definition of equality in Section \ref{sec:ss} that has been used in alternating play CGT with good results, introduced in \cite{Plamb2005} but see \cite{PlambS2008} for a good introduction.

Table~\ref{tbl: subtraction squares} illustrates the differences between the sums defined in the next three sections. For examples pertaining to \textsc{simultaneous hackenbush}, see Section~\ref{case study hackenbush}.

\begin{table}[h]
\begin{center}
\begin{tabular}{|c||c|c|c|}\hline
&\underline{1}&\underline{2}&\underline{3} \\
\hline\hline
$\,^{\mathcal{L}}G$&$\{\,\underline{0}\,\}$&$\{\, \underline{1}\, \}$&$\{\, \underline{2}\, \}$\\
$\,G^{\mathcal{R}}$&$\emptyset$&$\{\, \underline{0}\, \}$&$\{\, \underline{1}\, \}$\\
$\,^{\mathcal{L}}G^{\mathcal{R}}$&$\emptyset$&$\{\, \underline{0}\, \}$&$\{\, \underline{0}, \, \underline{1}\, \}$\\\hline
\end{tabular}
\caption{Summary of options for positions of \SQ($\{1\}, \{2\}$).}\label{tbl: subtraction squares}
\end{center}
\end{table}
 \subsection{Disjunctive Sum}\label{disjunctive sum}

Players are moving at the same time, and 
   thus all play combinations across components must be considered when analyzing any particular game under 
   disjunctive sum. By the above value assignments, the expected value will enable us to determine a local 
   expectation for a player to win a particular game. Left prefers positive expected value, and Right prefers
    negative expected value. However, if the expected value is zero, it does not imply a Draw. Similarly, if an 
    expected value is positive it does not guarantee that Left will win. 
   
\begin{definition}\label{normal disjunctive}
The \emph{disjunctive sum} of two combinatorial games being played under simultaneous moves, means that each player chooses a component and plays a legal move in that component. 
 Formally, the set of options from $G+H$ are as follows: 
\[G+H =  
\begin{cases}
\{\,^{L}G + H^{R}, \, ^{L}G^{R}+H, \,G^{R} + \,^{L}H, \, G + \,^{L}H^{R}\},\\
\quad\quad~\text{if $^{\mathcal{L}}G$ or $^{\mathcal{L}}H$ and $G^{\mathcal{R}}$ or $H^{\mathcal{R}}$ are non-empty};\\
\emptyset,\quad~\text{otherwise}.
\end{cases}
\]
Using Table~\ref{tbl: subtraction squares}, consider $G = \underline{1} + \underline{2} + \underline{3}$. The game is not terminal because Left has moves in $\underline{1}$, $\underline{2}$, and $\underline{3}$, while Right has moves in $\underline{2}$ and $\underline{3}$. \\

\noindent
The \textit{extended normal winning convention} is: if $G+H = \emptyset$, then
\begin{itemize}
\renewcommand{\labelitemi}{{\bf $\circ$}}
\item  Left wins if she has a move remaining, in $G$, $H$, or both, but Right does not. 
\item Right wins if he has a move remaining, in $G$, $H$, or both,  but Left does not. 
\item Otherwise the game is a Draw. 
\end{itemize}
\end{definition}

In normal play, CGT, $v(G+H)=v(G)+v(H)$. For simultaneous play, the hope would be that 
$Ex(G+H) = Ex(G)+Ex(H)$ or something equally simple for $Ex(G+H)$. 
However, the presence of $\LG + \HR$ and 
$\GR+\LH$ in the options of $G+H$ makes this unlikely for all but a few games
as the next example shows. 

Note, we will use $Ex(\underline{n})$ as shorthand for $Ex(\SQ(\{1\}, \{2\})(\underline{n}))$.

\begin{example}\label{ex:dsnp}
   In \SQ($\{1\}, \{2\}$), the position $\underline{2}$ is 
   a Draw and 
   $Ex(\underline{2}) = 0$    but now consider $\underline{2}+\underline{2}$. Right can never win so
   $Ex(\underline{2}+\underline{2})\geq 0$. In the analysis, $l\,\underline{2}$ and $r\,\underline{2}$ denotes playing in the left or right component
   in $\underline{2}+\underline{2}$. Playing in the same component always results in $\underline{0}$, regardless, so
   we can abbreviate $M(\underline{2}+\underline{2})$. Thus, $Ex(\underline{2}+\underline{2}) = 1/2$ (see Figure~\ref{expected values ss}) which is not $Ex(\underline{2})+Ex(\underline{2})$ or $Ex(\underline{2})$$Ex(\underline{2})$.

  \begin{figure}[h]
  \begin{center}
 \begin{minipage}[b]{0.28\linewidth}
      \[
\begin{blockarray}{ccc}
&l\,\underline{2}& r\,\underline{2} \\
\begin{block}{c[cc]}
 l\, \underline{2}& \underline{2} & \underline{1}\\
  r\, \underline{2} & \underline{1}& \underline{2}\\
\end{block}
\end{blockarray}
 \]
 \caption{Games.}
  \end{minipage}
    \begin{minipage}[b]{0.28\linewidth}
      \[
\begin{blockarray}{ccc}
&l\,\underline{2} & r\,\underline{2} \\
\begin{block}{c[cc]}
 l\,\underline{2}& D & L \\
  r\,\underline{2}& L & D \\
\end{block}
\end{blockarray}
 \]
 \caption{Outcomes.}
  \end{minipage}
    \begin{minipage}[b]{0.3\linewidth}
      \[
\begin{blockarray}{ccc}
&l\,\underline{2} & r\,\underline{2} \\
\begin{block}{c[cc]}
 l\,\underline{2}& 0 & 1 \\
  r\, \underline{2} & 1 & 0 \\
\end{block}
\end{blockarray}
 \]
 \caption{Expected Values.}\label{expected values ss}
  \end{minipage}
  \end{center}
\end{figure}

     \end{example}  
    
Example \ref{ex:dsnp} shows that the obvious test for $G=H$ in simultaneous, extended normal play cannot be just $Ex(G) = Ex(H)$
even though this is a necessary condition. 

\begin{question} 
Is there a set of conditions which only involve followers of $G$ and $H$ to
prove that $G\geq H$ in simultaneous, extended normal play with disjunctive sum? 
\end{question}

\subsection{Conjunctive Sum}\label{conjunctive sum}

\begin{definition}\label{normal conjunctive}
The \emph{conjunctive sum} of two simultaneous combinatorial games, written $G\wedge H$,
means that each player plays a legal move in all components.  Formally, the set of options from 
$G\wedge H$ are as follows: 
\[G\wedge H=  \{\,^{L}(G\wedge H)^{R}\} =
\begin{cases}
\{^{L}G^{R}\wedge \,^{L}H^{R}\}, ~\text{if $\LGR$~\emph{and}~$\LHR$~ are non-empty}; \\
\emptyset, ~\text{otherwise}.\\
\end{cases}\]

If $K_{1}\wedge K_{2}\wedge\ldots \wedge K_{n}= \emptyset$, then this means that at least one of the components is terminal. The extended normal play winning convention becomes:
\begin{itemize}
\renewcommand{\labelitemi}{{\bf $\circ$}}
\item  Left wins if she has an option in every terminal component, but Right does not. 
\item Right wins if he has an option in every terminal component, but Left does not. 
\item Otherwise the game is a Draw. 
\end{itemize}
\end{definition}
Using Table~\ref{tbl: subtraction squares}, consider now $G = \underline{1} \wedge \underline{2} \wedge \underline{3}$. The game is over because Right does not have a move in $\underline{1}$. \\

The game $a\wedge b$ finishes when one component finishes. This brings in a timing issue and one cannot expect $Ex(a\wedge b)$ to be a simple combination of $Ex(a)$ and $Ex(b)$. 

As a direct example consider \SQ$'(\{1\},\{2\}$), in the game $\underline{5} \wedge \underline{6}$ (see Figure~\ref{example: 5 and 6}). If Right's move does not overlap with Left's in $\underline{5}$, this guarantees that Right will win the game $\underline{5} \wedge \underline{6}$. However, Right cannot control this. After one turn, the resulting position will be one of the following: 
\begin{itemize}
\renewcommand{\labelitemi}{{\bf $\circ$}}
\item $\underline{2} \wedge \underline{3}$ or $\underline{2} \wedge \underline{4}$, and the game is over and Right wins;
\item $\underline{3} \wedge \underline{3}$, $Ex(\underline{3} \wedge \underline{3}) = 1/4$;
\item $\underline{3} \wedge \underline{4}$, $Ex(\underline{3} \wedge \underline{4}) = 1/4$.
\end{itemize}

 Note that $Ex(\underline{5}) = -0.25$, $Ex(\underline{6}) = 0.25$ and $Ex(\underline{5}\wedge \underline{6}) = -0.25$ and thus, in general, $Ex(a\wedge b) \neq Ex(a)Ex(b)$ and $Ex(a\wedge b) \neq Ex(a) + Ex(b)$.
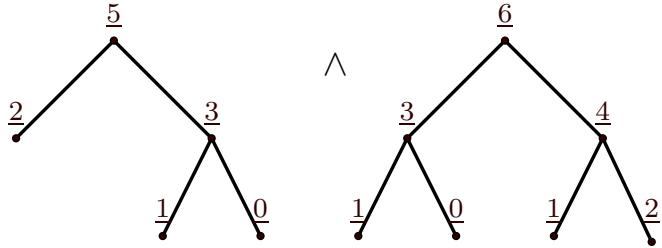
\begin{figure}[H]
\begin{center}
\definecolor{ttqqqq}{rgb}{0.2,0.,0.}
\scalebox{1.3}{
\begin{tikzpicture}[line cap=round,line join=round,>=triangle 45,x=1.0cm,y=1.0cm]
\draw [line width=1.pt] (4.,2.)-- (3.,1.);
\draw [line width=1.pt] (4.,2.)-- (5.,1.);
\draw [line width=1.pt] (5.,1.)-- (4.5,0.);
\draw [line width=1.pt] (5.,1.)-- (5.5,0.);
\draw [line width=1.pt] (8.,2.)-- (7.,1.);
\draw [line width=1.pt] (8.,2.)-- (9.,1.);
\draw [line width=1.pt] (7.,1.)-- (6.5,0.);
\draw [line width=1.pt] (7.,1.)-- (7.5,0.);
\draw [line width=1.pt] (9.,1.)-- (8.5,0.);
\draw [line width=1.pt] (9.,1.)-- (9.5,-0.06);
\draw (6,2) node[anchor=north west] {$\wedge$};
\begin{scriptsize}
\draw [fill=ttqqqq] (4.5,0.) circle (1pt);
\draw[color=ttqqqq] (4.5,0.25) node {$\underline{1}$};
\draw [fill=ttqqqq] (5.5,0.) circle (1pt);
\draw[color=ttqqqq] (5.5,0.25) node {$\underline{0}$};
\draw [fill=ttqqqq] (3.,1.) circle (1pt);
\draw[color=ttqqqq] (3,1.25) node {$\underline{2}$};
\draw [fill=ttqqqq] (5.,1.) circle (1pt);
\draw[color=ttqqqq] (5,1.25) node {$\underline{3}$};
\draw [fill=ttqqqq] (4.,2.) circle (1pt);
\draw[color=ttqqqq] (4.,2.25) node {$\underline{5}$};
\draw [fill=ttqqqq] (6.5,0.) circle (1pt);
\draw[color=ttqqqq] (6.5,0.25) node {$\underline{1}$};
\draw [fill=ttqqqq] (7.5,0.) circle (1pt);
\draw[color=ttqqqq] (7.5,0.25) node {$\underline{0}$};
\draw [fill=ttqqqq] (7.,1.) circle (1pt);
\draw[color=ttqqqq] (7,1.25) node {$\underline{3}$};
\draw [fill=ttqqqq] (8.5,0.) circle (1pt);
\draw[color=ttqqqq] (8.5,0.25) node {$\underline{1}$};
\draw [fill=ttqqqq] (9.5,-0.06) circle (1pt);
\draw[color=ttqqqq] (9.5,0.25) node {$\underline{2}$};
\draw [fill=ttqqqq] (9.,1.) circle (1pt);
\draw[color=ttqqqq] (9.,1.25) node {$\underline{4}$};
\draw [fill=ttqqqq] (8.,2.) circle (1pt);
\draw[color=ttqqqq] (8,2.25) node {$\underline{6}$};
\end{scriptsize}
\end{tikzpicture}}
\caption{\textsc{subtraction squares}: $\underline{5} \wedge \underline{6}$.}\label{example: 5 and 6}
\end{center}
\end{figure}

\subsection{Continued Conjunctive Sum}\label{continued conjunctive sum}

\begin{definition}\label{normal continued conjunctive}
The \emph{continued conjunctive sum} of two simultaneous combinatorial games, written $G\bigtriangledown H$,
means that each player plays a legal move in each component where they both have a move.  
Formally, the set of options from $G\bigtriangledown H$ are as follows: 
\[G\bigtriangledown H=  \{^{L}(G\bigtriangledown H)^{R}\} =
\begin{cases}
\{^{L}G^{R}\, \bigtriangledown\, ^{L}H^{R}\}, ~\text{if $\LGR$~\emph{or}~$\LHR$~ are non-empty}; \\
\emptyset, ~\text{otherwise}.\\
\end{cases}\]

Using Table~\ref{tbl: subtraction squares}, consider $G = \underline{1} \bigtriangledown \underline{2} \bigtriangledown \underline{3}$. The game is not over because both players have moves in $\underline{2}$ and $\underline{3}$.\\

If $G\bigtriangledown H = \emptyset$, then
\begin{itemize}
\renewcommand{\labelitemi}{{\bf $\circ$}}
\item  Left wins if she has an option in every component, but Right does not. 
\item Right wins if he has an option in every component, but Left does not. 
\item Otherwise the game is a Draw. 
\end{itemize}
\end{definition}

\begin{theorem}
Let $G = G_{1}\bigtriangledown G_{2}\bigtriangledown\ldots \bigtriangledown G_{n}$. If any of the components, $G_{i}$, is a Draw, then $G$ is also a Draw. 
\end{theorem}
\begin{proof}
This result follows immediately from the definition of continued conjunctive sum. 
\end{proof}

The expected values of $G$ and $H$ are not enough to determine the expected value of $G\bigtriangledown H$. 
 For example, when playing \SQ$'(\{1,4\},\{2\})(\underline{4})$ we obtain the results in Figures~\ref{ex: continued1} and~\ref{ex: continued2}. Note that $x_{l}$ means that the player removed $x$ from the left side of $\underline{4}$. 

  \begin{figure}[H]
  \begin{center}
 \begin{minipage}[b]{0.28\linewidth}
 \[
\begin{blockarray}{ccc}
&$$2_{l}$$&$$2_{r}$$ \\
\begin{block}{c[cc]}
$$ 1_{l}$$&$ $ \underline{2}$$ &$$\underline{1}$$\\
$$ 1{r}$$&$ $ \underline{1}$$ &$$\underline{2}$$\\
$$ 4_{l}$$&$ $ \underline{0}$$ &$$\underline{0}$$\\
$$ 4_{r}$$&$ $ \underline{0}$$ &$$\underline{0}$$\\
\end{block}
\end{blockarray}
\]
\caption{Options for \SQ$'(\{1,4\},\{2\})(\underline{4})$.}\label{ex: continued1}
 \end{minipage}
 \begin{minipage}[b]{0.3\linewidth}
 \[
\begin{blockarray}{ccc}
&$$2_{l}$$&$$2_{r}$$ \\
\begin{block}{c[cc]}
$$ 1_{l}$$&$ $ -1$$ &$$1$$\\
$$ 1_{r}$$&$ $ 1$$ &$$-1$$\\
$$ 4_{l}$$&$ $0$$ &$$0$$\\
$$ 4_{r}$$&$ $ 0$$ &$$0$$\\
\end{block}
\end{blockarray}\]
\caption{Expected Values.}\label{ex: continued2}
 \end{minipage}
 \end{center}
 \end{figure} 
 \noindent The expected value of \SQ$'(\{1,4\},\{2\})(\underline{4})$ is $0$.
 Now consider \SQ$'(\{1,4\},\{2\})(\underline{3})$. We obtain,

   \begin{figure}[H]
  \begin{center}
 \begin{minipage}[b]{0.28\linewidth}
 \[\begin{blockarray}{ccc}
&$$2_{l}$$&$$2_{r}$$ \\
\begin{block}{c[cc]}
$$ 1_{l}$$&$ $ \underline{1}$$ &$$\underline{0}$$\\
$$ 1{r}$$&$ $ \underline{0}$$ &$$\underline{1}$$\\
\end{block}
\end{blockarray}\]
\caption{Options for \SQ$(\{1,4\},\{2\})(\underline{3})$.}
 \end{minipage}
 \begin{minipage}[b]{0.3\linewidth}
\[\begin{blockarray}{ccc}
&$$2_{l}$$&$$2_{r}$$ \\
\begin{block}{c[cc]}
$$ 1_{l}$$&$ $ 1$$ &$$0$$\\
$$ 1{r}$$&$ $0$$ &$$1$$\\
\end{block}
\end{blockarray}\]
\caption{Expected Values.}
 \end{minipage}
 \end{center}
 \end{figure} 

 In $\underline{4}$, let $(p_1, p_2,p_3,p_4)$ be the probabilities of  playing
  $1_\ell,1_r,4_\ell,4_r$ respectively. 
Playing  $(0,0,\frac{1}{2}, \frac{1}{2})$ and $(\frac{1}{2}, \frac{1}{2},0,0)$ have expected values of 0,
but playing  $(0,0,\frac{1}{2}, \frac{1}{2})$ gives  $Ex(\underline{4}\bigtriangledown \underline{3}) = 0$,  whereas 
$(\frac{1}{2}, \frac{1}{2},0,0)$ gives $Ex(\underline{4}\bigtriangledown \underline{3}) = \frac{1}{4}$.

Even worse, dominated  strategies for a game $G$ may be the best for $G\bigtriangledown H$. For example, let $H$ be a terminal game in which
$\LH=\emptyset$ and $\HR\ne \emptyset$ and let $G$ have the expected
values
\[\begin{blockarray}{ccccc}
&$a$&$b$&$c$&$d$\\
\begin{block}{c[cccc]}
$x$&1&$-1$&$-1/2$&$1/4$\\
$y$&$-1$&$1$&$1/4$&$-1/2$\\
\end{block}
\end{blockarray}\]
where, further, Right has no chance of winning in the positive options, and Left has no chance of of winning in the negative
options.  Clearly Left plays $(x,y)= (1/2,1/2)$. The expected value of $G$ is $-1/4$ and is achieved when $(a,b,c,d) = (0,0,1/2,1/2)$. Right playing 
$(a,b,c,d) = (1/2,1/2,0,0)$ has expected value 0. However, in $G\bigtriangledown H$, Left
 can never win since Left always loses in $H$.
Therefore, Right achieves the expected value $-1/2$ with $(a,b,c,d) = (1/2,1/2,0,0)$.

This example shows that in evaluating $G\bigtriangledown H$, we cannot replace $G$ and $H$ by their reduced forms,
that is, $Ex(G\bigtriangledown H)\ne Ex(Re(G)\bigtriangledown Re(H))$.\\

In order to test for inequality, we need to know more about the probabilities of winning, drawing and losing rather than simply
 the expected values of $G$ and $H$. We expand on this in Section \ref{sec:ss} where we consider the games $SQ(\{a\},\{b\})$.

\section{Sums of Simultaneous Games: Scoring Play}\label{sums of simultaneous games: scoring}

For the purposes of this section: Left wins if $score_{G}>0$, Right wins if $score_{G} <0$, and otherwise the game is a Draw. 

\subsection{Disjunctive sum for Scoring Play}

\begin{definition}\label{scoring disjunctive}
The \emph{disjunctive sum} of two simultaneous combinatorial games under 
 scoring play, means that each player chooses a component and plays a legal move in that component. 
 Formally, the set of options from $G+H$ are as follows: 
\[G+H =  
\begin{cases}
\{\,^{L}G + H^{R}, \, ^{L}G^{R}+H, \,G^{R} + \,^{L}H, \, G + \,^{L}H^{R}\},\\
\quad\quad~\text{if $^{\mathcal{L}}G$ or $^{\mathcal{L}}H$ and $G^{\mathcal{R}}$ or $H^{\mathcal{R}}$ are non-empty};\\
score_{G+H},\quad~\text{if $G+H$ is terminal}.
\end{cases}
\]
\end{definition}

 \begin{example}\label{dead-end}
 Consider the games $A= \{-5\mid\cdot\}$ and $B = \{\cdot\mid7\}$. Under simultaneous play, $A$ is a Left win, since she has an option but Right does not. Similarly, $B$ is a Right win since Right has an option but Left does not. So $A \geq B$.
  Now consider $A+ B$ under disjunctive sum. On the first turn, Left plays in $A$ to $-5$ and Right plays in $B$ to $7$. From here, Right will run out of moves before Left, and thus, the game $A + B$ is a Left win. Hence, individual components do not tell us what will happen in a disjunctive sum.  
 \end{example}
 
If $^{\mathcal{L}}G = \emptyset$ and $G^{\mathcal{R}} = \emptyset$, then $G = 0$ and the game is a Draw. However, the converse is not true. For example, let $G=$\SQ$(\{1\},\{2\})(\underline{2})$ then
 $Ex(G) = 0$ and the game is a Draw, but both Left and Right have moves. 

\subsection{Conjunctive sum for Scoring Play}

\begin{definition}\label{scoring conjunctive}
The \emph{conjunctive sum} of two simultaneous combinatorial games under scoring play, written $G \wedge H$, means that each player plays a legal move in all components. Once one component is a terminal position, the game ends and the score for that component is the score for the game.  Formally, the set of options from $G\wedge H$ are as follows: 
\[G\wedge H=  \{^{L}(G\wedge H)^{R}\} =
\begin{cases}
\{^{L}G^{R}\wedge \,^{L}H^{R}\}, ~\text{if $^{\mathcal{L}}G^{\mathcal{R}}$~\emph{and}~$^{\mathcal{L}}H^{\mathcal{R}}$~ are non-empty}; \\
\text{score}_{G}, ~\text{if $^{\mathcal{L}}G$ or $G^{\mathcal{R}} = \emptyset$~and~$^{\mathcal{L}}H^{\mathcal{R}}$~is non-empty};\\
\text{score}_{H},~\text{if $^{\mathcal{L}}H$ or $H^{\mathcal{R}} = \emptyset$~and~$^{\mathcal{L}}G^{\mathcal{R}}$~is non-empty};\\
\text{score}_{G}+\text{score}_{H}, ~\text{if $G \wedge H$ is terminal.}\\
\end{cases}\]
\end{definition}

\subsection{Continued Conjunctive sum for Scoring Play}
\begin{definition}\label{scoring continued conjunctive}
The \emph{continued conjunctive sum} of two simultaneous combinatorial games under scoring play, written $G \bigtriangledown H$, means that each player plays a legal move in all components where both players have moves. Formally, the set of options from $G\bigtriangledown H$ are as follows: 
\[G\bigtriangledown H=  \{^{L}(G\bigtriangledown H)^{R}\} =
\begin{cases}
\{^{L}G^{R}\bigtriangledown \,^{L}H^{R}\}, ~\text{if $\LGR$~\emph{or}~$\LHR$~ are non-empty}; \\
\text{score}_{G}+\text{score}_{H}, ~\text{if $G \bigtriangledown H$ is terminal}.\\
\end{cases}\]
\end{definition}

Note: It is only necessarily true that $score_{G+H} = score_{G}+score_{H}$ under continued conjunctive sum. \\

\begin{theorem}\label{thm: expected value scoring} Let $G$ and $H$ be simultaneous combinatorial games. Then

\begin{enumerate}
\renewcommand{\labelitemi}{{\bf $\circ$}}
\item $E( G\bigtriangledown H) = Ex(G)+Ex(H)$
\item $G \geq H$ if $Ex(G)\geq Ex(H)$.
\end{enumerate}
\end{theorem}
\begin{proof}
In $G\bigtriangledown H$, play in the two games is independent and the final score in each is counted for the sum. Hence 
$Ex(G\bigtriangledown H ) = Ex(G)+Ex(H)$.

Suppose $G\geq H$ then, for all $X$, we have
\begin{eqnarray*}
Ex(G\bigtriangledown X) &=& Ex(G)+Ex(X)\geq Ex(H\bigtriangledown X)=Ex(H)+Ex(X)
\end{eqnarray*}
and hence $Ex(G)\geq Ex(H)$.

If $Ex(G)\geq Ex(H)$ then the same inequalities hold.
\qed
\end{proof}

\begin{corollary}\label{cor:ccspo}
The quotient of simultaneous games played with the continued conjunctive sum and with the scoring convention is a total order that can be embedded in the rationals.\end{corollary}

\begin{proof} Two games are equal if they have the same expected value, therefore the equivalence classes are indexed by the common expected value. 
A terminal game has value of $-1$, $0$ or $1$ and thus the expected value of any game is a rational number.
Moreover, $G\geq H$ if $Ex(G)\geq Ex(H)$ thus the quotient forms a total order and the values are a subset of the rational numbers.
\qed
\end{proof}

It seems likely that the quotient is isomorphic to the rationals but we do not know of actual games to show this.

 
 \section{Case Studies}\label{sec: case studies}
 
      We provide three case studies. 
      
First we look at \textsc{simultaneous clobber}, a dicot game (games where both players can  move from every
  non-empty subposition). Now, all simultaneous dicot games studied under extended normal play, under any of the three sums are 
  Draws and are therefore trivial. However, if we consider a
   different metric, we can continue to study dicot games under simultaneous moves. One interpretation is
    to assign a value to one player's actions, as exemplified in this study.
         
\textsc{simultaneous hackenbush}, the second case study, has properties which allow for easy computation of $v_A(G)$ for particular restricted graph classes.
     
     We analyze \textsc{subtraction squares}, specifically $SQ(\{a\}, \{b\})$ on general strips, but in this case we do not consider sums. The game is described in terms of a new measure, more general than the expected value.


\subsection{\textsc{simultaneous clobber}}
\bigskip
\noindent
Ruleset for \textsc{simultaneous clobber}.
\begin{itemize}
\renewcommand{\labelitemi}{{\bf $\circ$}}
\item Board: A finite graph, where each vertex is occupied by either an X or an O. \\
\item Players: Left and Right, who move simultaneously.\\
\item Moves: On a move, a player clobbers one of their opponents' adjacent pieces. Left is assigned X. Right is assigned O.  If players choose to clobber their opponent's piece which their opponent is also using to clobber theirs, both pieces disappear. \\
\end{itemize}

For example, $[OX]$ played simultaneously, after one move becomes $[\, \, \, ]$. If it was defined simply as a placement swap then the game would be \emph{loopy} (both players could insist on only choosing that move and the game would never end). 

 The scoring variant that we consider here is the number of $O$'s clobbered. This is an asymmetric game since the best Right can do is hope for a Draw. So we know that the outcome classes are restricted to Left wins and Draws.

First, we looked at \textsc{simultaneous clobber} played on the complete graph on $n$ vertices, $K_{n}$, where each vertex has an $O$ except for one which has an $X$. There are two possibilities: Left and Right choose matching vertices, and hence the game goes to zero. This can happen in $(n-1)$ ways. Or they don't match in their choices (i.e., Left clobbers one of Right's pieces and a different piece of Right takes the place of Left's piece which moved). This can happen in $(n-1) \times (n-2)$ ways. Hence

\begin{theorem}
The expected value of \textsc{simultaneous clobber} with one piece for Left on $K_{n}$ is defined by the following recurrence relation with initial value $K_{2} = 0$: 
\begin{align*}
Ex( K_{n}) &= \frac{1}{n-1}\left(0\right) + \frac{n-2}{n-1}\left(1+ Ex(K_{n-1})\right)\\
\end{align*}
which implies for $n\geq 2$,
\begin{align*}
 Ex(K_{n}) &=\frac{n}{2}-1 .
\end{align*}
\end{theorem}

\begin{table}
\begin{center}
\begin{tabular}{|c|c|}\hline
Positions& Values \\\hline
$[ \ldots OOXOO\ldots]$&$(1+\sqrt{5})/4$\\\hline
$[\ldots OOXXOO \ldots]$&$1/2$\\\hline
$[\ldots OOXO]$&$(-1+\sqrt{5})/2$\\\hline
$[\ldots OOX]$&$0$\\\hline
\end{tabular}
\caption{Values for some \textsc{simultaneous clobber} positions.}\label{ClobberValues}
\end{center}
\end{table}
Next, we look at this game on an infinite path, starting with one piece for Left $[ \ldots OOXOO \ldots]$. Then we place two Left pieces adjacent to one another, $[\ldots OXXO \ldots]$. The analysis of subsequent positions which involve increasing the distance between Left pieces and where she only has two pieces on the infinite path, are left to the reader. Preliminary values are shown in Table~\ref{ClobberValues}. 

Consider the \textsc{simultaneous clobber} position $G = [\ldots OOXOXOO\ldots]$. We label the options as $A, B, C, D,$ and $E$, where $B$ and $D$ correspond to the two $X$'s (left to right) and $A$, $C$, and $E$ correspond to the $O$'s (left to right) alternating between the $X$'s. Then $G^{\mathcal{R}} = \{A_{R}, C_{L}, C_{R}, E_{L}\}$, where $A_{R}$ means that Right can move the $O$ in position $A$ to the right. Similarly $\,^{\mathcal{L}}G = \{B_{L}, B_{R}, D_{L}, D_{R}\}$. Under simultaneous moves, we obtain the matrix\footnote{We evaluate $M'(G)$ since we have replaced positions by their expected values.} in Figure~\ref{ex: disjunctive}.

 \begin{figure}[H]
  \begin{center}
 \begin{minipage}[b]{0.45\linewidth}
\[
\begin{blockarray}{ccccc}
&A_{R}&C_{L}&C_{R}&E_{L}\\
\begin{block}{c[cccc]}
B_{L}&\frac{-1+\sqrt{5}}{2}&\frac{3}{2}&1&1\\
B_{R}&\frac{3}{2}&0&0&1\\
D_{L}&1&0&0&\frac{3}{2}\\
D_{R}&1&1&\frac{3}{2}&\frac{-1+\sqrt{5}}{2}\\
\end{block}
\end{blockarray}
\]
\caption{$M'([\ldots OOXOXOO\ldots])$. }\label{ex: disjunctive}
\end{minipage}
 \begin{minipage}[b]{0.45\linewidth}
\[
\begin{blockarray}{ccc}
&A'_{R}&D'_{L}\\
\begin{block}{c[cc]}
B'_{L}&0&1\\
C'_{R}&1&0\\
\end{block}
\end{blockarray} + 1
\]
\caption{$M'(B_{L}C_{L})$.}\label{ex: disjunctive2}
\end{minipage}
\end{center}
\end{figure}

Consider $B_{L}C_{L} = [\ldots OX] + [XO\ldots] +1$. Since $1$ is a score, Left and Right are playing in $[\ldots OX] + [XO\ldots]$, and we need to find its expected value and add $1$. If we calculate  $Ex(B_{L}C_{L})$ by using the reduced values (see Table~\ref{ClobberValues}) we find the expected value is $1$. But actually calculating the expected value of the disjunctive sum, we obtain expected value $\frac{3}{2}$ (see Figure~\ref{ex: disjunctive2}); i.e., 
 \begin{eqnarray}
 Ex(G+H) &\neq& Ex(Re(G) + Re(H)).
 \end{eqnarray}
This exemplifies once again that under disjunctive sum, we encounter problems with using previously defined `values' in a different sum.

We now examine a combinatorial game under simultaneous moves which allows the development of interesting results. 

\subsection{\textsc{simultaneous hackenbush}}\label{case study hackenbush}

\noindent
Ruleset for \textsc{simultaneous hackenbush}.
\begin{itemize}
\renewcommand{\labelitemi}{{\bf $\circ$}}
\item Board: A finite graph, where the edges are coloured either blue, red or green with a special set of root vertices connected to the ground. \\
\item Players: Left and Right, who move simultaneously.\\
\item Moves: Left can remove a blue or green edge, Right can remove a red or a green edge. 
After a simultaneous move, any connected component no longer connected to the ground is also deleted.
\\
\end{itemize}
 Note that in the figures, blue edges are represented by solid straight lines while red are dashed lines. 

If players move in the same component, they independently remove their chosen edge and all sub-graphs which are disconnected from the ground are eliminated.

\begin{figure}[H]
\definecolor{ffqqtt}{rgb}{1.,0.,0.2}
\definecolor{ffqqqq}{rgb}{1.,0.,0.}
\definecolor{qqqqff}{rgb}{0.,0.,1.}
\definecolor{xfqqff}{rgb}{0.4980392156862745,0.,1.}
\definecolor{ttqqqq}{rgb}{0.2,0.,0.}
\definecolor{uuuuuu}{rgb}{0.26666666666666666,0.26666666666666666,0.26666666666666666}
\begin{tikzpicture}[line cap=round,line join=round,>=triangle 45,x=1.0cm,y=1.0cm]
\clip(-4.3,-1) rectangle (16,2.5);
\draw [line width=2.pt,color=qqqqff] (0.,1.)-- (0.,0.);
\draw [line width=2.pt,color=qqqqff] (1.,1.)-- (1.,0.);
\draw [line width=2.pt,dash pattern=on 5pt off 5pt,color=ffqqqq] (0.,2.)-- (0.,1.);
\draw [line width=2.pt,dash pattern=on 5pt off 5pt,color=ffqqqq] (1.,2.)-- (1.,1.);
\draw (0.25,1.14) node[anchor=north west] {$\odot$};
\draw [line width=2.pt,color=qqqqff] (3.,0.)-- (3.,1.);
\draw [line width=2.pt,dash pattern=on 5pt off 5pt,color=ffqqtt] (3.,1.)-- (3.,2.);
\draw [line width=2.pt,color=qqqqff] (4.,1.)-- (4.,2.);
\draw [line width=2.pt,color=qqqqff] (4.,0.)-- (4.,1.);
\draw (3.25,1.14) node[anchor=north west] {$\odot$};
\draw [line width=2.pt] (-0.5,-0.04)-- (1.5,-0.04);
\draw [line width=2.pt] (2.52,-0.04)-- (4.6,-0.02);
\draw (-1,-0.34) node[anchor=north west] {\textbf{$G=$}};
\draw (-0.25,-0.34) node[anchor=north west] {\textbf{$G_{1}$}};
\draw (0.25,-0.34) node[anchor=north west] {\textbf{$\odot$}};
\draw (0.75,-0.34) node[anchor=north west] {\textbf{$G_{2}$}};
\draw (2.05,-0.32) node[anchor=north west] {\textbf{$H=$}};
\draw (2.75,-0.32) node[anchor=north west] {\textbf{$H_{1}$}};
\draw (3.25,-0.32) node[anchor=north west] {\textbf{$\odot$}};
\draw (3.75,-0.32) node[anchor=north west] {\textbf{$H_{2}$}};
\begin{scriptsize}
\draw [fill=uuuuuu] (0.,0.) circle (1.5pt);
\draw [fill=ttqqqq] (0.,1.) circle (1.5pt);
\draw [fill=ttqqqq] (0.,2.) circle (1.5pt);
\draw [fill=ttqqqq] (1.,0.) circle (1.5pt);
\draw [fill=ttqqqq] (1.,1.) circle (1.5pt);
\draw [fill=ttqqqq] (1.,2.) circle (1.5pt);
\draw[color=qqqqff] (-0.26,0.67) node {$x$};
\draw[color=qqqqff] (0.74,0.67) node {$y$};
\draw[color=ffqqqq] (-0.26,1.67) node {$z$};
\draw[color=ffqqqq] (0.74,1.67) node {$w$};
\draw [fill=ttqqqq] (3.,0.) circle (1.5pt);
\draw [fill=ttqqqq] (3.,1.) circle (1.5pt);
\draw [fill=ttqqqq] (3.,2.) circle (1.5pt);
\draw [fill=ttqqqq] (4.,0.) circle (1.5pt);
\draw [fill=ttqqqq] (4.,1.) circle (1.5pt);
\draw [fill=ttqqqq] (4.,2.) circle (1.5pt);
\draw[color=qqqqff] (2.76,0.61) node {$x'$};
\draw[color=ffqqtt] (2.7,1.69) node {$z'$};
\draw[color=qqqqff] (3.76,1.65) node {$w'$};
\draw[color=qqqqff] (3.76,0.69) node {$y'$};
\end{scriptsize}
\end{tikzpicture}
\caption{\textsc{simultaneous red-blue hackenbush} positions.}\label{example: disjunctive sum}
\end{figure}
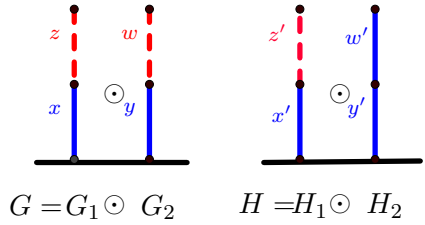

\begin{example}
The disjunctive sum of $G_{1}$ and $G_{2}$, two \textsc{simultaneous red-blue hackenbush stalks}, is shown in Figure~\ref{example: disjunctive sum}. Note that $^{\mathcal{L}}G = \{\,^{x}G,\,^{y}G\}$ and $G^{\mathcal{R}} = \{G^{z}, G^{w}\}$. Hence $\,^{\mathcal{L}}G^{\mathcal{R}} = \left\{\,^{x}G^{z},\,^{y}G^{z},\,^{x}G^{w},\,^{y}G^{w}\right\}$. We can represent simultaneous play in a matrix. Note that if Left plays $x$ and Right plays $z$, then $y$ and $w$ remain as options for the next round, and, considering outcomes, play is represented as: 

\begin{figure}[H]
\begin{center}
\begin{minipage}[b]{0.28\linewidth}
  \[
\begin{blockarray}{cc}
&$$G^{w}$$\\
\begin{block}{c[c]}
  $$\,^{y}G$$ & $0$ \\
\end{block}
\end{blockarray}
 \] 
 \caption{Game.}
 \end{minipage}
 \begin{minipage}[b]{0.28\linewidth}
   \[
\begin{blockarray}{cc}
&$$G^{w}$ $ \\
\begin{block}{c[c]}
 $$\,^{y}G$$ &$ D$ \\
\end{block}
\end{blockarray}
\]
\caption{Outcome.}
 \end{minipage}
 \end{center}
 \end{figure}

Similar results hold if the roles of $x$ and $y$ as well as $z$ and $w$ were interchanged.  Hence, recursively, the game $G$ is represented by the  matrices of Figures~\ref{fig: outcomes G}, \ref{fig: expected values G}, and \ref{fig: scores G}, for its outcomes, expected values, and score, respectively.

\begin{figure}[H]
\begin{center}
 \begin{minipage}[b]{0.25\linewidth}
  \[
\begin{blockarray}{ccc}
&$$G^{z}$$ & $$G^{w}$$  \\
\begin{block}{c[cc]}
 $$\,^{x}G$$ & D & L \\
  $$\,^{y}G$$ & L & D \\
\end{block}
\end{blockarray}
 \]
 \caption{Outcomes.}\label{fig: outcomes G}
 \end{minipage}
\hspace{0.5cm}
 \begin{minipage}[b]{0.3\linewidth}
  \[
\begin{blockarray}{ccc}
&$$G^{z}$$ &$$G^{w}$$  \\
\begin{block}{c[cc]}
 $$\,^{ x}G$$ & 0 & 1 \\
  $$\,^{y}G$$ & 1 & 0 \\
\end{block}
\end{blockarray}
 \]
 \caption{Expected Values.}\label{fig: expected values G}
 \end{minipage}
  \hspace{0.5cm}
 \begin{minipage}[b]{0.25\linewidth}
  \[
\begin{blockarray}{ccc}
&$$G^{z}$$ &$$G^{ w}$$  \\
\begin{block}{c[cc]}
  $$\,^{x}G$$ & 0 & 1 \\
 $$\,^{ y}G$$ & 1 & 0 \\
\end{block}
\end{blockarray}
 \]
 \caption{Scores.}\label{fig: scores G}
 \end{minipage}

 \end{center}
\end{figure}
\end{example}

Similarly, for $H$ we have, outcomes, expected values and scores are as shown in Figures~\ref{fig: outcomes H}, \ref{fig: expected values H}, and \ref{fig: scores H}, respectively. 

\begin{figure}[h]
\begin{center}
 \begin{minipage}[b]{0.25\linewidth}
  \[
\begin{blockarray}{cc}
&$$H^{z'}$$  \\
\begin{block}{c[c]}
 $$\,^{ x'}H$$ & L  \\
 $$\,^{y'}H$$ & L  \\
 $$\,^{ w'}H$$ &L\\
\end{block}
\end{blockarray}
 \]
 \caption{Outcomes.}\label{fig: outcomes H}
 \end{minipage}
\hspace{0.5cm}
 \begin{minipage}[b]{0.3\linewidth}
  \[
\begin{blockarray}{cc}
&$$H^{z'}$$  \\
\begin{block}{c[c]}
 $$ \,^{x'}H$$ & 1 \\
  $$\,^{y'}H$$ & 1 \\
  $$\,^{w'}H$$&1\\
\end{block}
\end{blockarray}
 \]
  \caption{Expected Values.}\label{fig: expected values H}
 \end{minipage}
  \hspace{0.5cm}
 \begin{minipage}[b]{0.25\linewidth}
  \[
\begin{blockarray}{cc}
&$$H^{z'}$$\\
\begin{block}{c[c]}
  $$\,^{x'}H$$ & 2 \\
 $$\, ^{y'}H$$ & 1\\
$$\,^{w'}H$$&2\\
\end{block}
\end{blockarray}
 \]
   \caption{Scores.}\label{fig: scores H}
 \end{minipage}
 \end{center}
\end{figure}

\begin{example}\label{example: win lose}
Consider $G_{1} \wedge G_{2}$ pictured in Figure~\ref{lose quick}. On the first turn, Right can guarantee a win in $G_{1}$ by playing $G_{1}^{c}$. Left knows this and hence rather than losing the game by playing $^{f}G_{2}$, she will play $^{d}G_{2}$ and force the overall game to be a Draw rather than a Right win. 
\end{example}

We introduce cordons, which are rooted graphs slightly more complex than rooted paths.

\begin{definition}
A \textit{cordon} consists of i) two sets of vertices $V_{1} = \{v_{0}, v_{1},\ldots,v_{n}\}$, where $v_0$ is the \emph{root}, $v_n$ is the \emph{top vertex} and the others are called \textit{interior} vertices, 
and $V_{2} = \{l_{1},l_{2},\ldots,l_{k}\}$, ii) an increasing sequence $\{a(1), a(2),\ldots,a(k)\}$, where $0<a(1), a(k)\leq n-1$, and iii) the edges are $v_{i}v_{i-1}$, $i = 1,2,\ldots, n$ and $l_{j}v_{a(j)}$, $j=1,\ldots ,k$. The vertex $v_{a(j)}$ is called an \emph{attachment vertex}. If $V_2$ is empty then we call the cordon a \emph{stalk}. See Figure~\ref{cordon} for an example of a cordon and $G_{1}$ and $G_{2}$, in Figure~\ref{lose quick}, for examples of stalks.
 \end{definition} 

 \begin{figure}[ht]
\begin{minipage}[b]{0.5\linewidth}
\centering
\definecolor{xfqqff}{rgb}{0.4980392156862745,0.,1.}
\definecolor{ffqqqq}{rgb}{1.,0.,0.}
\definecolor{qqqqff}{rgb}{0.,0.,1.}
\definecolor{ttqqqq}{rgb}{0.2,0.,0.}
\definecolor{uuuuuu}{rgb}{0.26666666666666666,0.26666666666666666,0.26666666666666666}
\begin{tikzpicture}[line cap=round,line join=round,>=triangle 45,x=1.0cm,y=1.0cm]
\draw [line width=2.pt,dash pattern=on 5pt off 5pt,color=ffqqqq] (4.,4.)-- (4.,3.);
\draw [line width=2.pt,color=qqqqff] (4.,3.)-- (4.,2.);
\draw [line width=2.pt,dash pattern=on 5pt off 5pt,color=ffqqqq] (4.,2.)-- (4.,1.);
\draw [line width=2.pt,color=qqqqff] (4.,1.)-- (4.,0.);
\draw [line width=2.pt,dash pattern=on 5pt off 5pt,color=ffqqqq] (0.,3.)-- (0.,2.);
\draw [line width=2.pt,color=qqqqff] (0.,2.)-- (0.,1.);
\draw [line width=2.pt,dash pattern=on 5pt off 5pt,color=ffqqqq] (0.,1.)-- (0.,0.);
\begin{scriptsize}
\draw [fill=uuuuuu] (4.,0.) circle (1.5pt);
\draw [fill=ttqqqq] (4.,1.) circle (1.5pt);
\draw [fill=ttqqqq] (4.,2.) circle (1.5pt);
\draw [fill=ttqqqq] (4.,3.) circle (1.5pt);
\draw [fill=ttqqqq] (4.,4.) circle (1.5pt);
\draw[color=ffqqqq] (4.26,3.67) node {$g$};
\draw[color=qqqqff] (4.26,2.67) node {$f$};
\draw[color=ffqqqq] (4.26,1.67) node {$e$};
\draw[color=qqqqff] (4.26,0.67) node {$d$};
\draw (2,2) node[anchor=north west] {$\bf{\wedge}$};
\draw [fill=uuuuuu] (0.,0.) circle (1.5pt);
\draw [fill=ttqqqq] (0.,1.) circle (1.5pt);
\draw [fill=ttqqqq] (0.,2.) circle (1.5pt);
\draw [fill=ttqqqq] (0.,3.) circle (1.5pt);
\draw[color=ffqqqq] (-0.26,2.67) node {$c$};
\draw[color=qqqqff] (-0.26,1.67) node {$b$};
\draw[color=ffqqqq] (-0.26,0.67) node {$a$};
\draw (-0.25,-0.5) node[anchor=north west] {$G_{1}$};
\draw (3.75,-0.5) node[anchor=north west] {$G_{2}$};
\end{scriptsize}
\end{tikzpicture}
\caption{Timing Issues.}\label{lose quick}
\end{minipage}
\hfill
\begin{minipage}[b]{0.4\linewidth}
\centering
\definecolor{ffqqqq}{rgb}{1.,0.,0.}
\definecolor{qqqqff}{rgb}{0.,0.,1.}
\definecolor{ttqqqq}{rgb}{0.2,0.,0.}
\begin{tikzpicture}[line cap=round,line join=round,>=triangle 45,x=1.0cm,y=1.0cm]
\draw [line width=2.pt,color=qqqqff] (0.,5.)-- (0.,4.);
\draw [line width=2.pt,color=qqqqff] (0.,4.)-- (0.,3.);
\draw [line width=2.pt,color=qqqqff] (0.,3.)-- (0.,2.);
\draw [line width=2.pt,color=qqqqff] (0.,2.)-- (0.,1.);
\draw [line width=2.pt,color=qqqqff] (0.,1.)-- (0.,0.);
\draw [line width=2.pt,dash pattern=on 5pt off 5pt,color=ffqqqq] (0.,1.)-- (1.,2.);
\draw [line width=2.pt,color=qqqqff] (0.,3.)-- (1.,4.);
\begin{scriptsize}
\draw [fill=ttqqqq] (0.,0.) circle (1.5pt);
\draw [fill=ttqqqq] (0.,1.) circle (1.5pt);
\draw [fill=ttqqqq] (0.,2.) circle (1.5pt);
\draw [fill=ttqqqq] (0.,3.) circle (1.5pt);
\draw [fill=ttqqqq] (0.,4.) circle (1.5pt);
\draw [fill=ttqqqq] (1.,2.) circle (1.5pt);
\draw [fill=ttqqqq] (1.,4.) circle (1.5pt);
\draw [fill=ttqqqq] (0.,5.) circle (1.5pt);
\end{scriptsize}
\end{tikzpicture}
\caption{A cordon.}\label{cordon}
\end{minipage}
\end{figure}

 \subsubsection{Results for Extended Normal Play}
 \begin{figure}[ht]
\begin{minipage}[b]{0.5\linewidth}
\centering
\definecolor{qqqqff}{rgb}{0.,0.,1.}
\definecolor{ttqqqq}{rgb}{0.2,0.,0.}
\definecolor{uuuuuu}{rgb}{0.26666666666666666,0.26666666666666666,0.26666666666666666}
\begin{tikzpicture}[line cap=round,line join=round,>=triangle 45,x=1.0cm,y=1.0cm]
\draw [line width=2.pt,color=qqqqff] (0.,0.)-- (0.,1.);
\draw [line width=2.pt,color=qqqqff] (0.,1.)-- (0.,2.);
\draw (-0.2,3.56) node[anchor=north west] {A};
\draw [rotate around={92.26065810668713:(-0.017047258424818752,3.3522410772208486)},line width=2.pt] (-0.017047258424818752,3.3522410772208486) ellipse (1.3530175815962302cm and 0.8782898812899246cm);
\begin{scriptsize}
\draw [fill=uuuuuu] (0.,0.) circle (2.0pt);
\draw [fill=ttqqqq] (0.,1.) circle (1.5pt);
\draw [fill=ttqqqq] (0.,2.) circle (1.5pt);
\end{scriptsize}
\end{tikzpicture}
\caption{Two-blue based position.}\label{two-blue based}
\end{minipage}
\hfill
\begin{minipage}[b]{0.4\linewidth}
\centering
\definecolor{ududff}{rgb}{0.30196078431372547,0.30196078431372547,1.}
\definecolor{qqqqff}{rgb}{0.,0.,1.}
\definecolor{ttqqqq}{rgb}{0.2,0.,0.}
\begin{tikzpicture}[line cap=round,line join=round,>=triangle 45,x=1.0cm,y=1.0cm]
\draw [line width=2.pt,color=qqqqff] (0.,1.)-- (0.,2.);
\draw (-0.08,3.56) node[anchor=north west] {A};
\draw [rotate around={92.26065810668713:(-0.017047258424818752,3.3522410772208486)},line width=2.pt] (-0.017047258424818752,3.3522410772208486) ellipse (1.3530175815962302cm and 0.8782898812899246cm);
\draw [line width=2.pt,color=ududff] (-0.24,4.22)-- (0.28,3.88);
\begin{scriptsize}
\draw [fill=ttqqqq] (0.,1.) circle (1.5pt);
\draw [fill=ttqqqq] (0.,2.) circle (1.5pt);
\draw [fill=ttqqqq] (-0.24,4.22) circle (1.5pt);
\draw [fill=ttqqqq] (0.28,3.88) circle (1.5pt);
\end{scriptsize}
\end{tikzpicture}
\caption{Blue* based position.}\label{blue* based position}
\end{minipage}
\end{figure}
We call a \textsc{simultaneous hackenbush} position which is rooted with one blue edge followed by anything else a \emph{blue-based} position. We call a position a \emph{blue* based position} if it is blue-based with at least one other edge somewhere else in the position (see Figure~\ref{blue* based position}). A \emph{two-blue based} position starts with two consecutive blue edges followed by anything above it, and no additional edge at $v_{1}$ (see Figure~\ref{two-blue based}). In both figures, A is a generic completion to the \textsc{simultaneous hackenbush} position.
 
\begin{lemma}\label{lem: outcomes}
Consider a \textsc{simultaneous hackenbush} position. If the first two edges are blue, followed by anything above it, then Left wins.
\end{lemma}
\begin{proof}
 Either all the edges are blue and Right has no move, or there are moves for Right, but Left can always remove the second edge from the bottom on the stalk which interferes with all of Right's moves (eliminating them). Then Left wins. 
 \qed
\end{proof}

Based on Lemma~\ref{lem: outcomes} we conclude the following results: 

\begin{proposition}
In a conjunctive sum of \textsc{simultaneous red-blue hackenbush} positions, where at least one component is two-blue based and all other components are blue* based,  then Left wins. 
\end{proposition}
\begin{proof}
Consider the first move within the conjunctive sum of \textsc{simultaneous red-blue hackenbush} positions satisfying the given properties. Within a two-blue based position (see Figure~\ref{two-blue based}), Left will remove the second blue edge from the root. This move guarantees that she will win this component since Right does not have a move on the next round and she does. Now, we need to ensure that Left has a move in all other components as well. Given the properties, there are at least two Left options in all other components: the bottom blue edge, and another edge somewhere else in the connected component. She chooses the latter option, to ensure that the components don't terminate (with her as the loser). In the conjunctive sum, she ends the overall game on the first move and is the winner.  \qed
\end{proof}

\begin{proposition}\label{prop: Right cannot lose}
For \textsc{simultaneous red-blue hackenbush stalks} which are purely alternating, starting with a blue edge and ending with a red edge, Right cannot lose. 
\end{proposition}
\begin{proof}
For every blue edge, there is a red edge directly above it. Right's strategy is to play the highest red edge available. When Left removes the last blue edge (rooted), Right will have an option to remove the red edge directly above it, the game is over and it is a Draw.\qed
\end{proof} 

\begin{proposition}
For \textsc{simultaneous red-blue hackenbush stalks} which are purely alternating, starting with a blue edge and ending with a blue edge, Right cannot use this component to force a Draw. 
\end{proposition}
\begin{proof}
Consider the induced subgraph on the vertices $\{v_{1}, \ldots,  v_{n}\}$. This is the negative of the position described in Proposition~\ref{prop: Right cannot lose}. Left will have one edge remaining after the game on the subgraph has terminated (Right has no move) and thus she willl win this component.
\end{proof}

\begin{proposition}
For \textsc{simultaneous red-blue hackenbush stalks} which starts by alternating and after alternation ends in two red edges, Right cannot lose this component.
\end{proposition}

\begin{proof}
Right can guarantee a Draw in this component by choosing the second red edge after alternation. Even if Left has chosen an edge above Right's choice on this round, the resulting position is either (i) as in Proposition~\ref{prop: Right cannot lose}, or (ii) starts and ends with red edges. In (i) by Proposition~\ref{prop: Right cannot lose} Right cannot lose. In (ii), consider the induced subgraph on $\left\{v_{1},\ldots,v_{n}\right\}$, this is as in Proposition~\ref{prop: Right cannot lose}. If Right ignores the edge connected to the ground $v_{0}v_{1}$, after simultaneous play ends in the subgraph, Right still has a move in the game (namely $v_{0}v_{1}$), and Left does not, and hence Right cannot lose. 
\end{proof}
\subsubsection{Results for Scoring Play}
\begin{definition}
The \emph{score} of a \textsc{simultaneous red-blue hackenbush} position is defined as the number of blue or red edges remaining after simultaneous play has ended. If there are $n$ blue edges remaining, the score of the position is $n$. If there are $n$ red edges remaining, the score  is $-n$. 
\end{definition}

\begin{lemma}\label{PlayFar}
Consider a \textsc{simultaneous  hackenbush stalk} with alternating blue and red edges. An optimal play has players moving furthest away from the ground. 
\end{lemma}
\begin{proof}
We prove this claim for alternating blue and red edges, starting with a blue edge. Symmetric proofs hold true if the stalk started with a red edge. There are two cases to consider: 1) ending with a blue edge; 2) ending with a red edge. Label the edges $l_{1}$, $\ldots$, $l_{n+1}$ for Left's options which $l_{1}$ being the edge closest to the ground, and $r_{1}$, $\ldots$, $r_{n}$ for Right's options. 
In both cases, consider their pure strategies as the labels of the following matrix rows and columns respectively.  

\underline{Case 1}:
 Applying the simultaneous moves recursively, the final matrix is the following $(n +1) \times n$ matrix:

\[
M=
 \begin{bmatrix}
   $0$&$0$&$0$&$\ldots$&$0$\\
   $1$&$0$&$0$&$\ldots$&$0$\\
   $1$&$1$&$0$&$\ldots$&$0$\\
   $\vdots$&&&&\\
   $1$&$1$&$1$&$\ldots$ &$1$
 \end{bmatrix}
\]

All pure strategies for Left are dominated by the final row ($l_{n+1}$) and hence the score of the game is $1$ and thus a Left win. 
 
\underline{Case 2}:
Applying the simultaneous moves recursively, the final matrix will be the following $n \times n$ matrix: 
\[
M=
 \begin{bmatrix}
   $0$&$0$&$0$&$\ldots$&$0$\\
   $1$&$0$&$0$&$\ldots$&$0$\\
   $1$&$1$&$0$&$\ldots$&$0$\\
   $\vdots$&&&&\\
   $1$&$1$&$1$&$\ldots$ &$0$
 \end{bmatrix}
\]
All pure strategies for Right  are dominated by the final column ($r_{n}$) and hence the score of the game is $0$ and thus is a Draw.
\qed
\end{proof}

Note: Lemma~\ref{PlayFar} does not necessarily hold in sums, as demonstrated in Example~\ref{example: win lose}. 

\begin{theorem}
The score of a \textsc{simultaneous red-blue hackenbush stalk} is the number, ${n}$, of blue (or red, respectively) edges before the first alternation between red and blue edges. If the alternation begins and ends with the same colour, then the score is $n$ (or $-n$ respectively). If the alternation begins and ends with different colours, then the score is ${ n-1}$ (or ${-n+1}$ respectively).
\end{theorem}

\begin{proof}
We show the proof for the stalk of score $n$ and $n-1$. It is a similar proof for $-n$ and $-n+1$. 

Consider a stalk where there are $n$ blue edges followed by a series of alternating red and blue edges, ending in two blue edges, followed by a string of $\alpha$ edges. There are six cases to consider:
\begin{itemize}
\renewcommand{\labelitemi}{{\bf $\circ$}}
\item Case 1: Both players move in $\alpha$. By induction, this game has value $n$.
\item Case 2: Left moves in $\alpha$, Right moves in the first alternating part. We are left with the position, of $n$ blue edges, and an alternating red-blue stalk above that, ending in blue. By Lemma~\ref{PlayFar}, both players will play their furthest edges and hence, in each turn the top two edges will be chosen. Right will run out of moves and Left will have $n$ edges remaining. 
\item Case 3: Right moves in $\alpha$ and Left moves in the first alternating part. The remaining stalk will have $n$ blue edges followed by alternating red-blue stalk above, ending in red. Again by Lemma~\ref{PlayFar}, both players will choose the furthest edges from the ground. This will result in $n-1$ blue edges at the end of simultaneous game play and thus is a dominated option (Case 1 and 2 are better options for Left). 
\item Case 4: Both players move in the first alternating part. By Lemma ~\ref{PlayFar}, both players will play at the top of this section of the stalk. Hence  We are left with the position, of $n$ blue edges, and an alternating red-blue stalk above that, ending in blue. Thus this falls into Case 2, and ends with a score of  $n$. 
\item Case 5 and 6: Left moving in the all blue string while Right moves in either $\alpha$ or the first alternating part. These options are dominated because it will result in a value less than $n$. 
\end{itemize}
\qed
\end{proof}

Under sequential play, the Conway values of alternating \textsc{red-blue hackenbush stalks} (starting with a blue edge) are approaching $2/3$ as the height of the stalk approaches infinity. Alternating \textsc{red-blue hackenbush stalks} (starting with a blue edge) with value less than $2/3$ are Draws in \textsc{simultaneous hackenbush}, while positions with values greater than $2/3$ are Left wins in \textsc{simultaneous hackenbush}. This fact, and Lemma~\ref{PlayFar}, lead us to the following conjecture.

\begin{conjecture}
Let $G$ be a \textsc{hackenbush} tree. If the CGT value of $G$ is greater than $2/3$, then $o_{S}(G)$ is a Left win.
\end{conjecture}

Similar results hold for Right if the roles of red and blue edges are interchanged.

\begin{lemma}
A \textsc{simultaneous red-blue hackenbush cordon} of height $n$ with all stalk edges blue and $a$ blue leaves and $b$ red leaves has score $n+a-b>0$.
\end{lemma}
Note: Interchanging the roles of red and blue edges, the score of the game is $-n-a+b<0$.
\begin{proof}

On a stalk, by Lemma~\ref{PlayFar}, we know players will play furthest away from the ground. Consider now a cordon, where all stalk edges are blue and there are $a$ blue leaves and $b$ red leaves. First let's consider Right's strategy. He only has leaves to play. If he chooses a leaf closer to the ground, and Left cuts a stalk edge below some red leaves, Right loses options. Hence Right will play leaves furthers away from the ground, to have minimal interference with Left. Left on the other hand, can either match Right's option by taking the stalk edge corresponding to the edge incident to a red leaf, or she could take a blue leaf. She is guaranteed all moves (since Right cannot interfere with any of Left's options), hence it is in her best interest to play leaves one at a time (including the nth stalk edge), and thus Left has a total of $n+a$ moves, and, if Left plays optimally, Right will have $b$ moves. Hence the value of the cordon is the number of extra moves Left has when Right runs out of moves, which is precisely $n+a-b$.  
\qed
\end{proof}

\subsection{\textsc{subtraction squares} $SQ(\{a\}, \{b\})$}\label{sec:ss}

\textsc{subtraction squares} $SQ(\{a\}, \{b\})$ is a special case of $SQ(S_{L}, S_{R})$ introduced in Section~\ref{ss:eando}. In this section, Left and Right each respectively only have one element in their subtraction set, $S_{L} = \{a\}$ and $S_{R}= \{b\}$. We use this special case to demonstrate a new concept we call the index.

\begin{definition} Let $G$ be a game and let: (i) $^{\ell}G$ be the game $G$ where every  terminal position
of value $-1$ (that is a Right win) is replaced by $0$; and (ii) $G^{r}$ be the game $G$ where every  terminal position
of value $1$ (that is a Left win) is replaced by $0$. Let $\ell_G=Ex(^{\ell}G)$ and $r_G=Ex(G^{r})$.
The \textit{index} of $G$ is $I(G) = \left[\ell_G, r_G\right]$.
 \end{definition}

Note that $\ell_G$ and $r_G$ are the probabilities that Left wins and Right wins respectively.
Intuitively, Left would prefer $G$ to $H$ in a sum if her chances of winning $G$ are at least as good as winning 
$H$ and if Right's chance were less. \\

\begin{theorem}
In the game $SQ(\{a\}, \{b\})$ under extended normal play, the expected value of $\underline{n}$ is given by
 \[Ex(\underline{n})=\begin{cases}
0, \text{if $n<a$;}\\
1, \text{if $a \leq n <b$;}\\
(Ex(\underline{n-b}) + Ex(\underline{n-b-a}))/2, \text{if $a+b \leq n$.}\\
 \end{cases}
 \]
\end{theorem}

\begin{proof}
If $n<a$ then neither player has a move and the game is a draw. If $a\leq n < b$ then only Left has a move and the game is a Left win. 
Suppose $n \geq b$. With equal probability the players play on the same side, leaving $n-b$ squares, or play on opposite sides leaving max$\{0,n-a-b\}$ giving $Ex(\underline{n}) =( Ex(\underline{n-b}) +Ex(\underline{n-b-a}))/2$.
\qed
\end{proof}

For given $a$, $b$, solving the game $SQ(\{a\},\{b\})$ means solving the recurrence  $2Ex(\underline{n}) = Ex(\underline{n-b}) +Ex(\underline{n-b-a})$ with the initial conditions. For example, in $SQ(\{1\},\{2\})$,

\begin{equation*}
Ex(\underline{n}) = \frac{1}{5}\left(2-\left(1+2i\right)\left(-\frac{1}{2} + \frac{i}{2}\right)^{n} - \left(1-2i\right)\left(-\frac{1}{2}-\frac{i}{2}\right)^{n}\right).
\end{equation*}

Since all the terms which are raised to the power $n$ are less than $1$ in modulus, then $lim_{n\rightarrow \infty} Ex(\underline{n})=2/5$. 

\begin{corollary}
For the subtraction game SQ$(\{a\},\{b\})$,  where $a < b$, $l_{\underline{n}}= Ex(\underline{n})$.
\end{corollary}
 \begin{proof}
Since Right cannot win, $r_{\underline{n}} = 0$. 
\qed
\end{proof}

Within the context of \textsc{subtraction squares} $SQ(\{a\}, \{b\})$, the following theory holds. 

\begin{theorem}\label{thm: index equality}
For simultaneous games $G$ and $H$ played with the continued 
conjunctive sum then 

(1)  $G=H$ iff $I(G)= I(H)$; and

 (2) if $G\geq H$ then $\ell_G\geq \ell_H$ and $r_G\leq r_H$.
\end{theorem}

\begin{proof} We only prove (2) since (1) is similar.

 Suppose that, for all games $X$, $I(G\bigtriangledown X)\geq I(H\bigtriangledown X)$. Setting $X=\emph{1}$,
 gives $I(G\bigtriangledown \emph{1}) =[\ell_G,0]$ and $I(H\bigtriangledown \emph{1}) =[\ell_H,0]$. It follows that
 $\ell_G\geq \ell_H$. Letting $X=\emph{-1}$ gives $r_G\leq r_H$. 
  \qed
\end{proof}

\begin{lemma}
Let $G$, $H$ be games then

(i)  $0\leq r_G,l_G\leq 1$ and $0\leq \ell_G+r_G\leq 1$;

(ii) $I(G\bigtriangledown  H) = [\ell_G\ell_H,r_Gr_H]$.

\end{lemma}

\begin{proof} For any game, $G$, since $\ell_G$ and $r_G$ are probabilities of mutually exclusive outcomes
 then $0\leq r_G,l_G\leq 1$  and $0\leq \ell_G+r_G\leq 1$.

Play in $G$ and play in $H$ are independent which gives $\ell_{G\bigtriangledown  H} = \ell_G\ell_H$.\qed
\end{proof}

In $SQ(\{a\},\{b\})$, both $\ell_{\underline{n}}$ and $r_{\underline{n}}$ are numbers. 
Can this be extended to larger subtraction sets? The immediate answer is no. The situation is more complicated since
$\ell_{\underline{n}}$ and $r_{\underline{n}}$ become functions. Consider
$SQ(\{1,4\},\{2\})$. The index for $\underline{4}$ is $[p,p]$, $0\leq p\leq 1/2$, moreover, the value of $p$ 
is determined by Left.

\section{Conclusions}

We extend combinatorial game theory with alternating play to allow simultaneous moves and develop the basic concepts
required to analyze these games. 
We then introduced and investigated three combinations of simultaneous games and two winning conventions. This 
included a strategic concept of equality and inequality. Since we are in the realm of two-player,  zero-sum games, 
dominated strategies can be eliminated without changing the expected value. In the disjunctive sum, under alternating play, 
the outcome of $G+H$ can be found by first reducing $G$ and $H$ then considering the sum of the resulting games.
However, for simultaneous play, we have shown that this only holds in the continued conjunctive sum under the scoring winning convention.

\begin{question}What sum, $\odot$, and reductions, can be applied to $G$ and $H$, giving $G'$ and $H'$ respectively,
 so that $Ex(G\odot H) =Ex(G'\odot H')$?
 \end{question}

 In the case studies, we examined a weaker form of equality and inequality, where positions from the same game 
 (for example) are compared. We formalize that approach. 
 Given a sum $\odot$, let an $\odot$-\textit{system}, $\mathbb{S}_{\odot}$ 
 be a set of positions closed under options and sums. That is, if  $G\in \mathbb{S}_{\odot}$ then (i) every position
obtainable from $G$ are also in $\mathbb{S}_{\odot}$; (ii) also,  if $H\in \mathbb{S}_{\odot}$ 
then $G\odot H\in \mathbb{S}_{\odot}$. Equality and inequality in $\mathbb{S}_{\odot}$, are given as in Definition \ref{defn:equality}, except now $X\in \mathbb{S}_{\odot}$. 

\begin{question} What $\mathbb{S}_{\odot}$ have reductions so that $Ex(G\odot H) =Ex(G'\odot H')$?
\end{question}
An interesting and important class of CGT games are the \textit{dead-ending} games, $\mathbb{D}$, which are defined by the property that if a player has no moves in a particular position then there is no sequence of moves that the opponent may make
that will allow the player to move again, see \cite{MilleR2013}. For example, in \textsc{domineering}, if there is no space for Left to place a vertical domino then allowing Right to place any number of horizontal dominoes will not create space for a vertical domino. \textsc{tridomineering} and \textsc{quadromineering} belong to this class.

\begin{question} For each sum, investigate $\mathbb{D}_{\odot}$.
\end{question}

One of the most important results within the theory of combinatorial games is that we understand how to sum games under different rulesets. Naturally, as we extend the theory to simultaneous play, we would like to have a similar theory developed here. For example, let 
\bigskip

\begin{center}$G = OXO \qquad \odot \qquad SQ'(\left\{1\right\},\left\{2\right\})$ on $\underline{4} \qquad\odot \quad{}$
\definecolor{ffqqqq}{rgb}{1.,0.,0.}
\definecolor{ttqqqq}{rgb}{0.2,0.,0.}
\begin{tikzpicture}[line cap=round,line join=round,>=triangle 45,x=1.0cm,y=1.0cm]

\draw [line width=2.pt, dash pattern=on 5pt off 5pt, color=ffqqqq] (0.,1.)-- (0.,0.);
\begin{scriptsize}
\draw [fill=ttqqqq] (0.,0.) circle (1.5pt);
\draw [fill=ttqqqq] (0.,1.) circle (1.5pt);
\end{scriptsize}
\end{tikzpicture}{\,}.
\end{center}
\bigskip

\begin{table}
\begin{center}
\begin{tabular}{|c|c|c|c|}\hline
& $+$&$\wedge$&$\bigtriangledown$ \\\hline
Extended Normal Play&R&R&D\\\hline
Scoring&-1/2&-1&-1/2\\\hline
\end{tabular}
\caption{Outcomes and expected values for $G$ based on different sums.}\label{RulesetSums}
\end{center}
\end{table}

Where do players want to move in $G$ under different sums and models? Table~\ref{RulesetSums} gives the game results (details are left to the reader). Initially, however, it is unclear which option is best given a particular sum and model. Ultimately, we would like to determine a method for combining sums of different rulesets to know the overall result for simultaneous play. 

\end{document}